\newif\ifalleinez
\ifx\Nmitglied\sdfasdfsdsdf\alleineztrue\else\alleinezfalse\fi
\ifalleinez

\documentclass[12pt,a4paper,twoside]{article}
\usepackage{pictex}
\usepackage{times}
\usepackage{contour}
\usepackage{amsthm,amsfonts,amsmath,amssymb}
\usepackage{imnvol}

\begin{document}
\def\Volume{224}
\def\Jahrgang{2013}
\month 12 \year 2013
\setcounter{page}{1}


 \def\issntext{}

\fi

 \def\boldx{{\mathord{\text{\sl\bfseries x}}}}
 \def\boldz{{\mathord{\text{\sl\bfseries z}}}}
 \def\boldPhi{\Phi}
 \def\boldinfty{\mathord{\text{\raise-.5ex\hbox{\LARGE${\infty}$}}}}

\english

\hyphenation{lo-kal-kom-pak-ter Dies-tel La-place Woess Lip-schitz}  

\numberwithin{equation}{section}

\newtheoremstyle{mythm}
  {9pt}
  {9pt}
  {\itshape}
  {0pt}
  {\bfseries}
  {}
  { }
  {\thmnumber{(#2)}\thmname{ #1}\thmnote{ #3}}

\newtheoremstyle{mydef}
  {9pt}
  {9pt}
  {\normalfont}
  {0pt}
  {\bfseries}
  {}
  { }
  {\thmnumber{(#2)}\thmname{ #1}\thmnote{ #3}}

\theoremstyle{mythm}
\newtheorem{thm}[equation]{Theorem.}
\newtheorem{pro}[equation]{Proposition.}
\newtheorem{lem}[equation]{Lemma.}
\newtheorem{cor}[equation]{Corollary.}
\newtheorem{clm}[equation]{Claim.}
\newtheorem{crt}[equation]{Criterion.}
\newtheorem{ques}[equation]{Question.}
\newtheorem{fct}[equation]{Facts.}

\def\Im{\mathop{\rm Im}\,}

\newtheorem{dfn}[equation]{Definition.}
\theoremstyle{mydef}
\newtheorem{rmk}[equation]{Remark.}

 \Artikel
	What is a horocyclic product, and \neuezeile
	how is it related to lamplighters?\neuezeile ~;
	Wolfgang Woess;
	TU Graz;

\insert\footins{\footnotesize
Supported by FWF (Austrian Science Fund) projects W1230-N13 and P24028-N18}

\markboth{{\sf W. Woess}}
{{\sf Horocyclic products and lamplighters}}
\baselineskip 15pt

{\it This is a rather personal introductory outline of an interesting class
of geometric, resp. graph- \& group-theoretical structures. After an 
introductive section about their genesis, the general construction of 
horocyclic products is presented. Three closely related basic structures of
this type are explained in more detail: Diestel-Leader graphs, treebolic spaces, 
and Sol-groups, resp. -manifolds. Emphasis is on their geometry, isometry
groups, quasi-isometry classification and boundary at infinity. Subsequently,
it is clarified under which parametrisation 
they admit discrete groups of isometries acting with
compact quotient. Finally, further develpoments are reviewed briefly.}

\section{A problem on infinite graphs}\label{sec:intro}

In the mid-1980ies, in conversations with my colleagues at Leoben,
I repeatedly asked the following question:

\begin{quote}
\emph{Are there any
vertex\dash transitive graphs that do not look like Cayley graphs?}
\end{quote}

The drawback was that I didn't see how to define ``look like''
rigorously. My eyes were opened when I encountered {\sc Gromov}'s 
definition of \emph{quasi\dash isometry} in \cite{Gr1}, resp.\ (more clearly)
\cite[7.2.G]{Gr2}. 

Before proceeding, we should clarify the involved notions and start a preliminary
discussion.
A \emph{graph} will be written in terms of its vertex set $X$, which carries a
symmetric neighbourhood relation $\sim$. Thus, the edges are pairs
$[x,y] = [y,x]$, where $x \sim y$, so that we allow loops $[x,x]$, but no
multiple edges. Usually, our graphs will be \emph{infinite}. 
The \emph{degree} $\deg(x)$ of $x \in X$ is the number of 
neighbours. Everbody is familiar with the concept of a \emph{path}
$[x_0, x_1\,, \dots, x_n]$ in a graph: one has to have  $x_k \sim x_{k-1}$,
and the length of a path is its number of edges (here: $n$). 

All our graphs 
will be \emph{connected} (for all $x, y \in X$ there is a path starting at $x$
and ending at $y$) and \emph{locally finite} ($\deg(x) < \infty$ for every $x$).
Being connected, $X$ becomes a metric space, where the graph distance $d(x,y)$ 
is the minimal length of a path from $x$ to~$y$. 

An \emph{automorphism} $X$ is a self\dash isometry of $(X,d)$. We write 
$\mathsf{Aut}(X)$ for the group of all automorphisms of $X$. The graph is called
\emph{vertex\dash transitive} if for every $x, y \in X$ there is $g \in \mathsf{Aut}(X)$
such that $gx=y$. A large class of such graphs is provided by groups:
given a finitely generated group $G$ (usually written multiplicatively)
and a finite, symmetric set $S$ of generators, we can visualise $G$ by its
\emph{Cayley graph} $X(G,S)$. Its vertex set is $X=G$, and $x \sim y$ if 
$y=xs$ for some $s \in S$ (so that $y = xs^{-1}$). The group acts by automorphisms
on $X(G,S)$ via $(g,x) \mapsto gx$. The most typical examples are

\begin{enumerate} \itemsep-\parsep
\item The Cayley graph of the additive group $\mathbb{Z}^2$ with respect to
$S = \{ (\pm 1,0)$, $(0,\pm 1)$ -- this is the square lattice;
\item The Cayley graph of the free group $\mathbb{F}_2$ on two free generators
$a$, $b$ with respect to $S = \{a^{\pm 1}, b^{\pm 1}\}$ -- this is the homogeneous
tree with degree $4$.
\end{enumerate}
See Figure~1. 
Furthermore, the homogeneous tree with arbitrary degree $p+1$ is also the
Cayley graph of the group $\langle a_1\,, \dots a_{p+1} : a_i^2 = 1_G \rangle$. 

\begin{figure}[h]
\hfill
\beginpicture  

\setcoordinatesystem units <.6mm,.6mm> 

\setplotarea x from -120 to 60, y from -30 to 28

\putrule from  -66 -20 to -110 -20
\putrule from  -66 -10 to -110 -10
\putrule from  -66 0 to -110 0
\putrule from  -66 10 to -110 10
\putrule from  -66 20 to -110 20
\putrule from  -68 -22 to -68 22
\putrule from  -78 -22 to -78 22
\putrule from  -88 -22 to -88 22
\putrule from  -98 -22 to -98 22
\putrule from  -108 -22 to -108 22

\put {$\mathbb{Z}^2$} [r] at -116 0

\putrule from  -4 24 to 4 24
\putrule from  -12 16 to 12 16
\putrule from  -20 8 to -12 8    \putrule from  20 8 to 12 8
\putrule from  -28 0 to 28 0
\putrule from  -20 -8 to -12 -8    \putrule from  20 -8 to 12 -8
\putrule from  -12 -16 to 12 -16
\putrule from  -4 -24 to 4 -24
\putrule from  24 -4 to 24 4
\putrule from  16 -12 to 16 12
\putrule from  8 -20 to 8 -12    \putrule from  8 20 to 8 12
\putrule from  0 -28 to 0 28
\putrule from  -8 -20 to -8 -12    \putrule from  -8 20 to -8 12
\putrule from  -16 -12 to -16 12
\putrule from  -24 -4 to -24 4

\put {$\mathbb{F}_2$} [l] at 34 0
\put {Figure~1} at 60 -20
\endpicture
\end{figure}

There are vertex\dash transitive graphs which are not Cayley graphs. 
A finite example is the well\dash known Petersen graph. From here one can of course construct 
infinite examples (e.g. the Cartesian product of the Petersen graph with the
bi\dash infinite line). But there also are intrinsically infinite examples of 
non\dash Cayley vertex\dash transitive graphs. One of them is based on the following way
of looking at trees, which will play an important role later on:
take the homogeneous tree $\mathbb{T}_p$ with degree $p+1$, but draw it differently,
such that it ``hangs down'' from a point $\varpi$ at infinity. See Figure 2, where $p=2$.
That is, the tree is considered as
the union of generations (horizontal layers) -- called \emph{horocycles} -- 
$H_k\,$, $k\in \mathbb{Z}$. Each $H_k$ is infinite, every vertex $x \in H_k$ has a unique 
neighbour in $H_{k-1}\,$, its \emph{predecessor} $x^-$, and $p$ neighbours in 
$H_{k+1}\,$, its \emph{successors}. Thus, $p$ is the \emph{branching number} of 
$\mathbb{T}$. For $x \in \mathbb{T}$, we write $\mathfrak{h}(x) = k$
if $x \in H_k\,$, the \emph{Busemann function.} An ancestor of $x$ is an iterated predecessor.
Any pair of vertices $x, y$ has a common ancestor $v$ for which $\mathfrak{h}(v)$
is maximal. We write $v=x \curlywedge y$. Also, we choose a root vertex $o$ in $H_0\,$.

\begin{figure}[h]
\hfill\beginpicture 

\setcoordinatesystem units <.55mm,.9mm> 

\setplotarea x from -10 to 104, y from 2 to 80

\arrow <6pt> [.2,.67] from 2 2 to 80 80

\plot 32 32 62 2 /

 \plot 16 16 30 2 /

 \plot 48 16 34 2 /

 \plot 8 8 14 2 /

 \plot 24 8 18 2 /

 \plot 40 8 46 2 /

 \plot 56 8 50 2 /

 \plot 4 4 6 2 /

 \plot 12 4 10 2 /

 \plot 20 4 22 2 /

 \plot 28 4 26 2 /

 \plot 36 4 38 2 /

 \plot 44 4 42 2 /

 \plot 52 4 54 2 /

 \plot 60 4 58 2 /

 \plot 99 29 64 64 /

 \plot 66 2 96 32 /

 \plot 70 2 68 4 /

 \plot 74 2 76 4 /

 \plot 78 2 72 8 /

 \plot 82 2 88 8 /

 \plot 86 2 84 4 /

 \plot 90 2 92 4 /

 \plot 94 2 80 16 /


\setdots <3pt>
\putrule from -4.8 4 to 102 4
\putrule from -4.1 8 to 102 8
\putrule from -2 16 to 102 16
\putrule from -1.7 32 to 102 32
\putrule from -1.7 64 to 102 64

\put {$\vdots$} at 32 -3
\put {$\vdots$} at 64 -3

\put {$\dots$} [l] at 103 6
\put {$\dots$} [l] at 103 48

\put {$H_{-3}$} [l] at -20 64
\put {$H_{-2}$} [l] at -20 32
\put {$H_{-1}$} [l] at -20 16
\put {$H_0$} [l] at -20 8
\put {$H_1$} [l] at -20 4
\put {$\vdots$} at -10 -3
\put {$\vdots$} [B] at -10 70
\put {$\varpi$} at 82 82

\put {$o$} at 7 10.5
\put {$x$} at 57.8 10.5
\put {$x^-$} at 53 19
\put {$y$} at 99 34.5
\put {$x \!\curlywedge\! y$} [r] at 63 66.5

\put {$\scriptstyle \bullet$} at 8 8
\put {$\scriptstyle \bullet$} at 56 8
\put {$\scriptstyle \bullet$} at 48 16
\put {$\scriptstyle \bullet$} at 96 32
\put {$\scriptstyle \bullet$} at 64 64

\put {Figure~2} at 160 0
\endpicture
\end{figure}

Consider the group $\mathsf{Aff}(\mathbb{T})$ of all automorphisms $g$ of 
$\mathbb{T}$ which preserve the predecessor relation: $g(x^-) = (gx)^-$ for all $x$. 
It acts transitively. It is called the affine group of the tree because it contains
the group of all affine mappings $\xi \mapsto \alpha\xi + \beta$ of the 
ring $\mathbb{Q}_p$ of $p$-adic numbers (field, if $p$ is prime), where 
$\xi, \alpha, \beta \in \mathbb{Q}_p$ and $\alpha$ is invertible,
see {\sc Cartwright, Kaimanovich and Woess}~\cite[\S 4]{CKW}. In particular, 
$\mathbb{Q}_p$ can be identified with the lower boundary $\partial^*\mathbb{T}_p\,$ 
that we are going to describe further below.  

Next, we introduce the additional edges $[x, (x^-)^-]$ for all $x$, 
see Figure~3. The resulting graph is sometimes called the \emph{grandmother graph}, which 
is suggestive when one thinks of $\mathbb{T}$ as an infinite genealogical tree.

\begin{figure}[h]
\hfill\beginpicture 

\setcoordinatesystem units <3mm,3mm> 

\setplotarea x from 4 to 12, y from 3 to 14

\plot 4 4  8 8  12 4  /

\plot 8 8  8 14 /

\setdashes <3pt>
 
\plot 4 4  8 14  12 4 /

\put {Figure~3} at 30 4
\endpicture
\end{figure}

The point is that $\mathsf{Aff}(\mathbb{T})$ now becomes the full automorphism
group of the grandmother graph.

\begin{clm}\label{claim:granma}
The grandmother graph is vertex\dash transitive, but not a Cayley graph of some finitely
generated group.
\end{clm}

How does one prove that a given graph is or is not a Cayley graph?

\begin{crt}\label{crit:Cayley}
Let $X$ be a locally finite, connected graph and $G$ be a subgroup
of $\mathsf{Aut}(X)$. Then $X$ is a Cayley graph of $G$ if and only
if $G$ acts on $X$ transitively and with trivial vertex\dash stabilisers.
\end{crt} 

Here the stabiliser of $x \in G$ is of course $G_x = \{ g\in G: gx =x\}$,
and ``trivial'' means that it consists only of the identity.

Now assume that a group of autmorphisms $G$ acts transitively on the grandmother
graph. Then $G \le \mathsf{Aff}(\mathbb{T})$. Let $x$ be a vertex and $y,z$ be
two of its successors. Then there must be $g \in G$ such that $gy = z$, so that
$g \neq \textrm{id}$. But we must have $gy^- = z^-$, that is $gx = x$. Thus,
$G_x$ is non\dash trivial, which proves Claim \ref{claim:granma}.

\smallskip

However, everybody will agree that the grandmother graph looks (vaguely) like
the tree itself, which \emph{is} a Cayley graph. So from the point of view of the 
initial question, this is not yet a satisfactory example.
Let us now come to the definition of ``look like''.

\begin{dfn}\label{def:qi}
Let $(X_1\,,d_1)$ and $(X_2\,,d_2)$ be two metric spaces.
A mapping $\varphi: X_1 \to X_2$ is called a \emph{quasi\dash isometry},
if there are constants $A > 0$ and $B \ge 0$ such that
for all $x_1\,,y_1 \in X_1$ and $x_2 \in X_2\,,$
\begin{enumerate}
\item[(i)] $d_2(x_2\,, \varphi X_1) \le B\quad$ (quasi\dash surjective), and
\item[(ii)] $\dfrac{1}{A}\, d_2(\varphi x_1\,,\varphi x_2) - B
\le d_1(x_1,y_1) \le A\, d_2(\varphi x_1\,,\varphi x_2) - B\;$ (quasi\dash bi\dash Lipschitz). 
\end{enumerate}
If $B=0$, the mapping is called \emph{bi\dash Lipschitz}.
\end{dfn}

Every quasi\dash isometry $\varphi$ has a quasi\dash inverse $\varphi^*: X_2 \to X_1\,$, 
i.e., a quasi\dash isometry such that $\varphi^*\, \varphi$ and $\varphi\, \varphi^*$
are bounded perturbations of the identity on $X_1\,$, resp.\ $X_2$ (i.e., the image
of any element is at bounded distance). In particular, quasi\dash isometry is
an equivalence relation. 

Any two Cayley graphs of a group with respect to different, finite symmetric sets
of generators are bi\dash Lipschitz. After being promoted by Gromov, the study of
quasi\dash isometry invariants of finitely generated groups has become a ``big business''
which is at the core of what is since then called Geometric Group Theory (in good
part replacing the earlier name ``Combinatorial Group Theory'').

The identity map on the vertex set is a bi\dash Lipschitz mapping between the grandmother 
graph and the tree, so that we have a non\dash Cayley vertex transitive graph which is
quasi-isometric with a Cayley graph. My question now could be formulated rigorously as
follows.

\begin{quote}
\emph{Is there a (connected, locally finite, infinite) vertex\dash transitive graph that is 
not quasi\dash isometric with some Cayley graph?}
\end{quote}    

I posed this question explicitly in \cite{SoWo} and \cite{W-top} (published 
in 1990, resp.\ 1991). 
This appeared to be a difficult  problem, and I learnt that there has to be a positive 
correlation between the difficulty of a mathematical question and the fame of 
the person who poses it. Initially, geometric group theorists
ignored my problem or even made fun of it. 
However, in the world of
Graph Theory, there is an exclusive minority interested in infinite graphs, and in 
the mid\dash early 1990ies, {\sc Diestel and Leader} came up with a construction of a 
graph which they believed to provide the answer to the question. This was what I 
later called the Diestel\dash Leader graph $\mathsf{DL}(2,3)$, whose construction will 
be explained in a moment. However, it resisted their and my efforts (as well as the 
efforts of several visitors of mine who were involved in this discussion) to prove 
that it was indeed not quasi\dash isometric with any Cayley graph.
At last, in 2001, Diestel and Leader made their construction and conjecture public
without a proof \cite{DiLe}.   

Let us now describe the construction. We take two trees $\mathbb{T}_p$ and
$\mathbb{T}_q$ with respective branching numbers $p$ and $q$ (not necessarily
distinct). We look at each of them as in Figure~2, but the second tree is
upside down. On each of them, we have the respective Busemann function $\mathfrak{h}$.
(We omit putting an index.) 

\begin{dfn}\label{def:DL}
The \emph{Diestel\dash Leader graph} $\mathsf{DL}(p,q)$ is
\begin{equation*}
\mathsf{DL}(p,q) 
= \{ (x_1\,,x_2) \in \mathbb{T}_p \times \mathbb{T}_q : 
\mathfrak{h}(x_1) + \mathfrak{h}(x_2) = 0 \},
\end{equation*}
and neighbourhood is given by
\begin{equation*}
(x_1\,,x_2) \sim (y_1\,,y_2) \;\Longleftrightarrow\; x_1 \sim y_1 \;\text{ and }\; x_2 \sim y_2\,. 
\end{equation*}
\end{dfn}
Thus, either $x_1^-=y_1$ and $y_2^- = x_2$ or vice versa.  

To visualize $\mathsf{DL}(p,q)$, draw $\mathbb{T}_p$ in horocyclic layers 
as in Figure~2, and right to it $\mathbb{T}_q$ in the same way, but upside down, 
with the respective horocycles $H_k(\mathbb{T}_p)$ and $H_{-k}(\mathbb{T}_q)$ on the 
same level. Connect the two origins $o_1$, $o_2$ by
an elastic spring. It is allowed to move along each of the two trees, 
may expand infinitely, but must always remain in horizontal position. 
The vertex set of $\mathsf{DL}(p,q)$ consists of all admissible positions of 
the spring. From a position
$(x_1\,,x_2)$ with $\mathfrak{h}(x_1) + \mathfrak{h}(x_2) =0$ the spring may
move downwards to one of the $q$ successors of $x_2$ in $\mathbb{T}_q\,$, and at the same 
time to the predecessor of $x_1$ in $\mathbb{T}_p\,$, or it may move upwards 
in the analogous way. Such a move corresponds
to going to a neighbour of $(x_1\,,x_2)$. Figure 2 depicts
$\mathsf{DL}(2,2)$.


\begin{figure}[h]
\hfill\beginpicture
\setcoordinatesystem units <3mm,3.5mm> 

\setplotarea x from -4 to 30, y from -10 to 6.4
\arrow <5pt> [.2,.67] from 4 4 to 1 7
\put{$\varpi_1$} [rb] at 1.2 7.2

\put{$o_1$} [lb] at  8.15 0.2

\plot -4 -4       4 4         /         
\plot 4 4         12 -4          /      
\plot -2 -2       -2.95 -4 /            
\plot -.5 -2      -1.9 -4     /         
\plot -.5 -2      -.85 -4   /           
\plot 1 -2        .2 -4      /          
\plot 1 -2        1.25  -4     /        
\plot 2.5 -2      2.3  -4     /         
\plot 2.5 -2      3.35 -4      /        
\plot 5.5 -2      4.65  -4    /         
\plot 5.5  -2     5.7  -4      /        
\plot 7  -2       6.75   -4   /         
\plot 7 -2        7.8  -4      /        
\plot 8.5  -2     8.85  -4    /         
\plot 8.5  -2     9.9  -4      /        
\plot 10  -2      10.95  -4   /         
\plot 0  0        -.5  -2      /        
\plot 2 0         1 -2     /            
\plot 2 0         2.5   -2     /        
\plot 6 0         5.5    -2    /        
\plot 6 0         7 -2         /        
\plot 8 0         8.5 -2       /        
\plot 2 2         2 0         /         
\plot 6 2         6 0         /         

\arrow <5pt> [.2,.67] from 22 -4 to 25 -7
\put{$\varpi_2$} [lt] at 25.2 -7.2

\put{$o_2$} [rt] at  17.95 -.2

\plot 14  4       22 -4       /         
\plot 22 -4         30 4         /      
\plot 16 2       15.05 4 /              
\plot 17.5 2     16.1  4     /          
\plot 17.5 2     17.15  4   /           
\plot 19  2       18.2  4      /        
\plot 19 2         19.25   4     /      
\plot 20.5 2       20.3   4     /       
\plot 20.5  2      21.35 4      /       
\plot 23.5 2       22.65  4    /        
\plot 23.5  2      23.7   4      /      
\plot 25   2       24.75   4   /        
\plot 25  2        25.8  4      /       
\plot 26.5   2     26.85  4    /        
\plot 26.5   2     27.9   4      /      
\plot 28   2      28.95   4   /         
\plot 18  0        17.5  2      /       
\plot 20 0         19  2     /          
\plot 20 0         20.5   2     /       
\plot 24 0         23.5    2    /       
\plot 24 0         25 2         /       
\plot 26 0         26.5 2       /       
\plot 20 -2        20 0         /       
\plot 24 -2        24 0         /       
\put {$\circ$} at 8 0
\put {$\circ$} at 18 0
\plot 8.25 0  12.1 0 /
\plot 13.9 0 17.78 0 /
\plot    12.1   0    12.25 .4    12.25 -.4   12.55 .4   12.55 -.4
        12.85 .4   12.85 -.4  13.15 .4   13.15 -.4  13.45 .4
        13.45 -.4  13.75 .4   13.75 -.4  13.9 0     13.9  0 /

\setdashes <2pt>
\putrule from -4.5 -7  to  12.5 -7
\putrule from  13.5 7  to  30.5 7


\put {$\vdots$} at 4 -5.2
\put {$\vdots$} at 22 5.5

\put {Figure~4} at 35 -8.5
\endpicture
\end{figure}

We first explain that $\mathsf{DL}(p,q)$ is vertex\dash transitive, but when
$p \ne q$, it is not a Cayley graph. Recall the group $\mathsf{Aff}(\mathbb{T})$ 
of all automorphisms that preserve the predecessor relation, where 
$\mathbb{T}= \mathbb{T}_p$ or $=\mathbb{T}_q\,$. One easily verifies (see
{\sc Cartwright, Kaimanovich and Woess}~\cite{CKW}) that the mapping
\begin{equation}\label{eq:Phi}
\Phi: \mathsf{Aff}(\mathbb{T}) \to \mathbb{Z}\,,\quad 
\Phi(g) = \mathfrak{h}(gx) - \mathfrak{h}(x)
\end{equation}
is independent of $x \in \mathbb{T}$, and thus a homomorphism onto the 
additive group $\mathbb{Z}\,$. That is, every $g \in \mathsf{Aff}(\mathbb{T})$
shifts the tree up or down by the vertical amount $\Phi(g)$. Now the following 
is not hard to prove.

\begin{pro}\label{pro:af}
The group 
$$
\mathcal{A} = \mathcal{A}(p,q) 
=\{ (g_1\,,g_2) \in \mathsf{Aff}(\mathbb{T}_p) \times \mathsf{Aff}(\mathbb{T}_q) :
\Phi(g_1) + \Phi(g_2) = 0 \}
$$ 
acts transitively on $\mathsf{DL}(p,q)$ by
$$
(x_1\,,x_2) \mapsto (g_1x_1\,, g_2x_2)\,.
$$
If $p \ne q$, then this is the full automorphism group of $\mathsf{DL}(p,q)\,$,
while when $p=q$, then it has index $2$ in the full automorphism group,
which is generated by $\mathcal{A}$ and the ``reflection'' $(x_1\,,x_2) \mapsto (x_2\,,x_1)$. 
\end{pro}
(The vertex set of) $\mathsf{DL}(p,q)$ is the disjoint union of the \emph{horoplanes} 
$$
H_{k,-k} = \{ (x_1\,,x_2) \in \mathbb{T}_p \times \mathbb{T}_q : \mathfrak{h}(x_1) = k\,,\;
\mathfrak{h}(x_2) = -k \}\,
$$
and every $g=(g_1\,,g_2) \in \mathcal{A}$ maps $H_{k,-k}$ to $H_{m,-m}\,$, where $m = k + \Phi(g_1)$.

\begin{lem}\label{lem:DLnoC}
If $q \ne p$ then $\mathsf{DL}(p,q)$ is not a Cayley graph of some finitely
generated group.
\end{lem}

\begin{proof}
Suppose that $q > p$, and that $G$ is any group of automorphisms that acts transitively
on $\mathsf{DL}(p,q)\,$. Consider the sets $A = \{ (o_1\,,x_2) : x_2^- = o_2^- \} \subset H_{0,0}$ and 
$B = \{ (x_1\,,o_2^-) : x_1^-= o_1\} \subset H_{1,-1}\,$. Then $|A| = q$, $|B| = p$, and 
the subgraph of $\mathsf{DL}(p,q)$ induced by $A \cup B$ is the complete bipartite graph over
$A$ and $B$ (there is an edge between each element of $A$ and each element of $B$). The set
$B$ consists of all neighbours of $A$ in $H_{1,-1}\,$.

For each $x = (o_1\,,x_2) \in A$ there must be $g_x \in G$ such that $g_xo=x$, where $o=(o_1\,,o_2)\,$.
Since $G \subset \mathcal{A}$, each $g_x$ sends every horoplane to itself, and preserves
neighbourhood. We conclude that each $g_x$ sends $B$ onto itself. But since $|B| <|A|$,
there  must be two distinct $x,x' \in A$ and $y \in B$ such that $g_xy = g_{x'}y$.
Then $g_x^{-1}g_{x'}$ stabilises $y$, although it is different from the identity.
In view of Criterion \ref{crit:Cayley},  $\mathsf{DL}(p,q)$ cannot be a 
Cayley graph of $G$.    
\end{proof}
The last proof gives a clue why $\mathsf{DL}(p,q)$ should not be quasi\dash isometric with
some Cayley graph, when $p \ne q$: briefly spoken, our graph grows on the order
of $p^n$ in one vertical direction, and of order $q^n$ in the opposite direction.

The result was finally announced in 2007 by a group of quasi\dash isometry experts, {\sc Eskin, Fisher
and Whyte}~\cite{ESW1}, and the proof is contained in the first of the two papers
\cite{ESW2}, \cite{ESW3} within a more general framework of quasi\dash isometry classification
of structures whose construction is very similar to DL-graphs. Thus, at last, my question 
made it to the {\em Annals}:

\begin{thm}{\rm\protect\cite{ESW2}}.\label{thm} If $q \ne p$ then $\mathsf{DL}(p,q)$ is not quasi\dash isometric
 with any finitely generated group. 
\end{thm}

\section{Horocyclic products}\label{sec:horo}

We now explain the first of the two notions of the title of this article. 
Let $X$ be a metric space.  A \emph{level function} or \emph{Busemann function}
is a continuous surjection $\mathfrak{h}: X \to \mathbb{L}$, where 
$\mathbb{L} = \mathbb{R}$, or when $X$ is discrete, resp.\ totally disconnected, 
$\mathbb{L} = \mathbb{Z}$. 
We write $H_l$ ($l\in\mathbb{L}$) for the associated level sets, i.e., 
the preimages under $\mathfrak{h}$ of $l$. We call them \emph{horocycles} or
\emph{horospheres}.

Usually, our $X$ will carry additional structure, and then the function $\mathfrak{h}$
should be adapted to that structure. If $X$ is a (connected) graph, then it
has to be a graph homomorphism (neighbourhood preserving surjection) onto 
$\mathbb{Z}$, the latter seen as the bi\dash infinite line graph. In particular, edges of
$X$ are only allowed between successive horocycles.

If $X=G$ is a discrete (or more generally, totally disconnected) group, then we will
need $\mathfrak{h}$ to be a
group homomorphism onto $\mathbb{Z}$, while if it is a connected locally compact
group, it has to be a (continuous) homomorphism onto $\mathbb{R}$. (More general
choices of Abelian groups $\mathbb{L}$ also work, but will not be considered here.)

We refer to $(X,\mathfrak{h})$ as a \emph{Busemann pair} over $\mathbb{L}$,
although this expression is justified only in specific cases.

\begin{dfn}\label{def:horprod}
Let $(X_1\,,\mathfrak{h}_1)$ and $(X_2\,,\mathfrak{h}_2)$ be two Busemann
pairs over the same $\mathbb{L}$. We shall commonly use the same symbol $\mathfrak{h}$
for both $\mathfrak{h}_i\,$. The \emph{horocyclic product} of
$X_1$ and $X_2$ is
$$
X_1 \times_{\mathfrak{h}} X_2 =
\{ (x_1\,,x_2) \in X_1 \times X_2 :  \mathfrak{h}(x_1) +  \mathfrak{h}(x_2) = 0 \}
$$
(On some occasions it may be more natural to require that  
$\mathfrak{h}(x_1) - \mathfrak{h}(x_2) = 0$.) 
\end{dfn}

In general, $X_1 \times_{\mathfrak{h}} X_2$ is 
a topological subspace of the direct product 
space $X_1 \times X_2\,$. In the group case, it is a normal subgroup of the direct product.
In the graph case, as edges in the $X_i$ may occur only between successive horocycles,
$X_1 \times_{\mathfrak{h}} X_2$ is an induced subgraph of the direct product of the two graphs. That is,
	\begin{align*}
	&(x_1,x_2) \sim (y_1\,,y_2) 
	\iff 
	x_i \sim y_i \; (i=1,2), \quad \text{and then}
	\\
	& 
	\mathfrak{h}(x_1) - \mathfrak{h}(y_1) = \mathfrak{h}(y_2) - \mathfrak{h}(x_2) = \pm 1\,.
	\end{align*}

\begin{rmk}\label{rmk:metric graphs} It may also be good to consider graphs as one\dash dimensional
complexes, where each edge is a copy of the unit interval. The graph metric extends naturally to the
interior points of the edges. If we have a $\mathbb{Z}$-valued Busemann function $\mathfrak{h}$ on the
vertex set, and $e =[x,y]$ is an edge with $x \in H_k$ and $y \in H_{k+1}\,$,
then we can extend $\mathfrak{h}$ to every interior point $z \in e$: if 
$d(z,x) = \kappa \in [0\,,\,1)$, then $\mathfrak{h}(z) = k + \kappa$. In this way, 
the Busemann function becomes $\mathbb{R}$-valued, and the topology of the 
resulting horocyclic product yields just the one\dash dimensional complex that one gets 
from the graph construction with edges $\equiv$ intervals.
\end{rmk}

\bigskip

\vbox{\noindent\textbf{Three sister structures} 

\smallskip\noindent
We now consider three families of spaces. The fact that they share many common 
geometric features becomes apparent by realising that they all are horocyclic
products:} 
\begin{enumerate}
\item[A.] horocyclic product of two trees $\to$ Diestel\dash Leader graphs;
\item[B.] horocyclic product of a hyperbolic half\dash plane and a tree $\to$ treebolic spaces;
\item[C.] horocyclic product of two hyperbolic half\dash planes $\to$ 
$\mathsf{Sol}$-groups, resp.\ manifolds.
\end{enumerate}
Thinking of a graph as a 1-complex as in Remark \ref{rmk:metric graphs}, 
our structures are $1$\dash dimensional in A, $2$\dash dimensional in B, and $3$\dash dimensional in C.

\bigskip

\noindent\textbf{A. More on Diestel\dash Leader graphs}

\smallskip\noindent
We start with some general observations.
The automorphism group of any locally finite, connected graph $X$ carries the topology of
pointwise convergence (on the vertex set), and as such, it is a locally compact,
totally disconnected group. See e.g. {\sc Trofimov}~\cite{Tr}, or \cite{W-top}. 
Let $G$ be any closed subgroup of
$\mathsf{Aut}(X)$ that acts transitively. It has a left Haar measure $\lambda_G\,$ 
(unique up to multiplication with a constant), and there is the \emph{modular function}
$\Delta_G$ defined by $\Delta_G(g) = \lambda_G(Ug)/\lambda_G(U)$, where $U \subset G$ 
is open with compact closure. $\Delta_G(g)$ is independent of the choice of $U$, and we
may take $U = G_x$, the stabiliser of some vertex $x$. The group is called unimodular 
when $\Delta_G \equiv 1$. 

\begin{lem}\label{lem_mod} {\rm \cite{Schl}, \cite{Tr}.}
If $g \in G$ and $gx = y$ then $\Delta_G(g) = |G_ox|/|G_xo|$. 
\end{lem}

A connected graph $X$ with bounded vertex degrees is called \emph{amenable},
if 
$$
\inf \{ |\partial F| / |F| : F \subset X \; \text{finite}\} = 0\,.
$$
A non\dash amenable graph is sometimes called \emph{infinite expander.}
A locally compact group is called \emph{amenable,} if it carries a left\dash invariant
mean $\mathsf{m}$, that is, a finitely additive measure that satisfies $\mathsf{m}(G)=1$
and $\mathsf{m}(gU) = \mathsf{m}(U)$ for any $g \in G$ and Borel set $U \subset G$.
The follwing is due to {\sc Soardi and Woess}~\cite{SoWo}.

\begin{pro}\label{pro:amen}
A vertex\dash transitive graph $X$ is amenable if and only if some ($\!\Longleftrightarrow\!$
every) closed subgroup $G$ of $\mathsf{Aut}(X)$ that acts transitively is
both amenable and unimodular.
\end{pro}

We also note that a horocyclic product of two amenable groups is amenable,
since it is a subgroup of the direct product of the two groups. (Known fact:
closed subgroups as well as direct products of amenable groups are amenable.)
Now it is easy to see and well\dash known that $\mathsf{Aff}(\mathbb{T}_p)$
is an amenable group, see e.g. \cite[Lemma 12.14]{Wbook}, and its modular
function is $\Delta_{\mathsf{Aff}(\mathbb{T}_p)}(g) = p^{\Phi(g)}$.
The group $\mathcal{A}(p,q)$ of Proposition \ref{pro:af} is the horocyclic product of
$\mathsf{Aff}(\mathbb{T}_p)$ and $\mathsf{Aff}(\mathbb{T}_q)$ with respect to the
respective Busemann functions $\mathfrak{h}(g) = \Phi(g)$, where $\Phi$ is given by
\eqref{eq:Phi}. 

It is easy to compute the modular function of the group $\mathcal{A}(p,q)$.

\begin{cor}\label{cor:Aqr}
The modular function of the group $\mathcal{A}(p,q)$ is given by
$$
\Delta_{\mathcal{A}}(g) = (q/r)^{{\boldPhi}(g)}\,,
$$
where ${\boldPhi}(g) = \Phi(g_1)=-\Phi(g_2)$ for $g=g_1g_2 \in \mathcal{A}(p,q)$.

Thus, the group $\mathcal{A}(p,q)$ is unimodular, and the graph
$\mathsf{DL}(p,q)$ is amenable if and only if $p=q$. 
\end{cor}

This leads to another view on the fact that $\mathsf{DL}(p,q)$ is not a Cayley graph
when $p \ne q$. Indeed, more generally, when $p \ne q$, then there cannot be
a finitely generated group of automorphisms that acts on $\mathsf{DL}(p,q)$ with 
finitely many orbits and finite vertex stabilisers: such a group would have to be a 
co\dash compact lattice, i.e., a discrete subgroup of $\mathcal{A}(p,q)$ with compact quotient, 
which cannot occur in a non\dash unimodular group.
Further below, we shall see that when $p=q$, the Diestel\dash Leader graph \emph{is}
a Cayley graph.

Regarding the quasi\dash isometry classification, we quote another result of
{\sc Eskin, Fisher and Whyte}.

\begin{thm}\label{thm:qi-DL} {\rm \cite{ESW1}+\cite{ESW2}}.
$\mathsf{DL}(p,q)$ is quasi\dash isometric with
$\mathsf{DL}(p',q')$ if and only if $p$ and $p'$ are powers of a common integer, 
$q$ and $q'$ are powers of a common integer, and 
$\log p'/ \log p = \log q'/\log q$.
\end{thm}

Another object whose description may be of interest is the geometric 
\emph{boundary at infinity} of $\mathsf{DL}(p,q)$. 

For that purpose, we first need
to describe the geometric boundary of an arbitrary infinite, locally finite tree 
$T$ (not necessarily homogenous).   For any $x, y$ in $T$,
there is a unique \emph{geodesic path} $\pi(x,y) = [x_0\,, \dots, x_n]$
such that $d(x_i\,,x_j) = |i-j|$ for all $i,j$. 
Analogously, a \emph{geodesic ray}, resp.\ \emph{(two\dash sided) geodesic} is
an infinite path $\pi = [x_0\,, x_1\,, x_2\,, \dots]$, resp.\
$\pi = [\dots, x_{-1}\,,x_0\,, x_1\,,x_2\,,\dots]$, such that  
$d(x_i\,,x_j) = |i-j|$ for all $i,j$. We think of a ray as a way of going 
to a point at infinity. Then two rays describe the same point at infinity,
i.e., they are equivalent, if their symmetric difference is finite. This means that
they differ only by finite initial pieces. An \emph{end} of $T$ is an 
equivalence class of rays. The boundary $\partial T$ is the set of all ends.
For any $x \in T$ and $\xi \in \partial T$, there is a unique geodesic ray
$\pi(x,\xi)$ that starts at $x$ and represents $\xi$. For any pair of distinct
ends $\xi, \eta$, there is a unique geodesic 
$\pi(\xi,\eta)= [\dots, x_{-1}\,,x_0\,, x_1\,,x_2\,,\dots]$ such that  
$[x_0\,,x_{-1}\,,x_{-2}\,,\dots]$ represents $\xi$ and
$[x_0\,,x_1\,,x_2\,,\dots]$ represents $\eta$. 

We choose a reference point $o \in T$ and let $|x| = d(o,x)$ for $x \in T$. 
For $w,z \in \widehat T = T \cup \partial T$,
we define their \emph{confluent} $w \wedge z$ with respect to $o$ by
$$
\pi(o, w \wedge z) = \pi(o,w) \cap \pi(o,z)\,.
$$ 
This is a vertex, namely the last common element on the geodesices 
$\pi(o,w)$ and $\pi(o,z)$, unless $w=z \in \partial T$.   
We equip $\widehat T$ with the following ultra\dash metric.
$$
\theta(w,z) = \begin{cases} e^{-|w \wedge z|}\,,\;&\text{if }\; z \ne w,\\
               0\,,\;&\text{if }\; z = w.
              \end{cases}
$$
Then $\widehat T$ is compact, and $T$ is open and dense. In the induced topology,
a sequence $z_n \in \widehat T$ converges to $\xi \in \partial T$ if and
only if $|z_n \wedge \xi| \to \infty$. 

Back to $\mathbb{T}_p$, we choose a reference end 
$\varpi \in \partial \mathbb{T}_p$ and let $\partial^*\mathbb{T}_p$ be the remaining
punctured boundary. In Figure~2, $\varpi$ is at the top and  $\partial^*\mathbb{T}_p$
at the bottom. The function $\mathfrak{h}$ is indeed the Busemann function
with respect to $\varpi$ in the classical sense: for any vertex $x$,
$$
\mathfrak{h}(x) = \lim_{y \to \xi} \bigl(d(x,y) - d(o,y)\bigr) = 
d(x,x\curlywedge o) - d(o,x\curlywedge o),
$$
where (recall) $x\curlywedge o$ is the maximal common ancestor of $x$ and $o$,
see Figure~2. (It is the confluent of $x$ and $o$ with respect to the end 
$\varpi$ instead of the vertex $o$.) 

In taking our two trees, we have two reference ends, 
$\varpi_1 \in \partial \mathbb{T}_p$ and $\varpi_2 \in \partial \mathbb{T}_q\,$,
see Figure~4. Now we can describe the natural geometric compactification of
$\mathsf{DL}(p,q)\,$: it is a subgraph of $\mathbb{T}_p\times \mathbb{T}_q\,$,
and the obvious geometric compactification of the latter product space is
$\widehat{\mathbb{T}}_p\times \widehat{\mathbb{T}}_q\,$.

\begin{dfn}\label{def:DLbdry}
The geometric compactification $\widehat{\mathsf{DL}}(p,q)$ is the closure
of $\mathsf{DL}(p,q)$ in $\widehat{\mathbb{T}}_p\times \widehat{\mathbb{T}}_q\,$, 
and the boundary at infinity is
$$
\partial \mathsf{DL}(p,q) = \widehat{\mathsf{DL}}(p,q) \setminus \mathsf{DL}(p,q).
$$
\end{dfn}
We can imagine the boundary as a ``filled ultra\dash metric 8''. It is
$$
\partial \mathsf{DL}(p,q) = \Bigl(\widehat{\mathbb{T}}_p\times \{\varpi_2\}\Bigr)
\cup \Bigl(\{\varpi_1\}\times \widehat{\mathbb{T}}_q\Bigr)\,.
$$
The two pieces meet in the point $(\varpi_1 \,,\varpi_2)$,
see Figure~5.

\begin{figure}[h]
\hfill\beginpicture  

\setcoordinatesystem units <.9mm,.9mm> 

\setplotarea x from -50 to 50, y from -27 to 27

\putrule from  -40 0 to -10 0
\putrule from  -25 -15 to -25 15
\plot  -38 -4  -35 0   -38 4 /
\plot  -12 -4  -15 0   -12 4 /
\plot  -29 13  -25 10   -21 13 /
\plot  -29 -13  -25 -10   -21 -13 /

\put {$\scriptstyle\rightarrow \varpi_1$} [r] at -0.9 -0.4

\putrule from  40 0 to 10 0
\putrule from  25 -15 to 25 15
\plot  38 -4  35 0   38 4 /
\plot  12 -4  15 0   12 4 /
\plot  29 13  25 10  21 13 /
\plot  29 -13  25 -10   21 -13 /

\put {$\scriptstyle \varpi_2 \leftarrow$} [l] at 1 -0.4

\setdashes <3pt>
\circulararc 360 degrees from 0 0 center at -25 0
\circulararc 360 degrees from 0 0 center at 25 0
\put {$\scriptstyle\bullet$} at 0 0

\put {$\widehat{\mathbb{T}}_p\times \{\varpi_2\}$} [r] at -50 10 
\put {$\{\varpi_1\}\times \widehat{\mathbb{T}}_q$} [l] at 50 10 

\put {Figure~5} at 68 -20
\endpicture
\end{figure}

Here, the topology of the boundary dictates disk\dash like pictures of the two trees
with their boundaries, while Figure~2 is an upper\dash half\dash plane\dash like picture. 

Let us clarify convergence to the boundary of a sequence 
${\boldx}_n =  (x_{1,n}\,,x_{2,n}) \in \mathsf{DL}(p,q)$ in the resulting topology.
At least one of $x_{1,n}$ and $x_{2,n}$ has to converge to a boundary point
of the respective tree. 
If $x_{1,n} \to \xi_1 \in \partial^* \mathbb{T}_p\,$, then necessarily 
$x_{2,n} \to \varpi_2$, whence ${\boldx}_n \to (\xi_1\,,\varpi_2)$.
Analogously, if $x_{1,n} = x_1 \in \mathbb{T}_p$ for all $n \ge n_0\,$, then
necessarily $x_{2,n} \to \varpi_2$, whence ${\boldx}_n \to (x_1\,,\varpi_2)$.
In the same way, when $x_{2,n} \to \xi_2 \in \partial^* \mathbb{T}_p\,$, resp.\
$x_{2,n}=x_2 \in \mathbb{T}_q$ for all $n \ge n_0\,$, then 
${\boldx}_n \to (\varpi_1\,,\xi_2)$, resp.\ ${\boldx}_n\to (\varpi_1\,, x_2)$.
Finally, it is possible that $x_{1,n} \to \varpi_1$ and $x_{2,n} \to \varpi_2$
(for example by staying on a fixed horizontal level). In this case,
${\boldx}_n \to (\varpi_1\,,\varpi_2)$. 

\smallskip

To conclude this description of the geometry of $\mathsf{DL}(p,q)$, we display
the formula for the graph metric, due to {\sc Bertacchi}~\cite{Ber}.

\begin{lem}\label{lem:daniela}
In $\mathsf{DL}(p,q)$, 
$$
d\bigl((x_1\,,x_2),(y_1\,,y_2)\bigr) = d(x_1\,,y_1) + d(x_2\,,y_2) 
- |\mathfrak{h}(x_1) - \mathfrak{h}(x_2)|.
$$
\end{lem}

\bigskip

\noindent\textbf{B. Treebolic spaces}

\smallskip\noindent

Let $\mathbb{H} = \{ z \in \mathbb{C} : \Im z > 0 \}$ be the
upper half plane with the hyperbolic metric
$$
d(z_1\,,z_2) = 
\log \frac{|z_1 - \overline z_2| + |z_1 - z_2|}{|z_1 - \overline z_2| - |z_1 - z_2|}\,.
$$
Recall that geodesics (shortest paths) lie on semi\dash circles orthogonal to the real axis, 
resp.\ vertical lines. The standard Busemann function with respect to the upper
boundary point ${\boldinfty}$ is $z \mapsto \log (\Im z)$.
In comparing with the tree, the sign is reversed -- it increases when going to 
${\boldinfty}$. (This is related with the fact that the real and $p$-adic 
absolute values of $p^n$, $n\in \mathbb{Z}$, have opposite behaviour.)
Now we rescale the Busemann function by choosing a  \emph{real} parameter $q > 1$ 
and setting $\mathfrak{h}(z) = \mathfrak{h}_q(z) = \log_q(\Im z)$.
Among the resulting horocycles, there are the ones where 
$\mathfrak{h}_q(z) = k \in \mathbb{Z}$, that is, $\Im z = q^k$. Drawing these
in the upper half plane yields to a picture to which we sometimes refer as
\emph{sliced hyperbolic plane} $\mathbb{H}_q\,$, see Figure~6.  

\begin{figure}[h]
\hfill\beginpicture 

\setcoordinatesystem units <1.3mm,.90mm>

\setplotarea x from -40 to 40, y from -8 to 80

\arrow <6pt> [.2,.67] from 0 0 to 0 77.5



\plot -40 4  40 4 /

\plot -40 8  40 8 /

\plot -40 16  40 16 /

\plot -40 32  40 32 /

\plot -40 64  40 64 /

\put {$i$} [rb] at -0.3 8.6
\put {$\scriptstyle \bullet$} at 0 8
\put {$y = q^{-1}$} [r] at -40 4
\put {$y = 1$} [r] at -43 8
\put {$y = q$} [r] at -43 16
\put {$y = q^2$} [r] at -42 32
\put {$y = q^3$} [r] at -42 64
\put {${\boldinfty}$} [t] at 0 80

\put {$\mathbb{R}$} [r] at -45 0

\setdashes <3pt>
\linethickness =.7pt
\putrule from -43.4 0 to 40.4 0 
\setlinear

\put {Figure 6} at 43 -6
\endpicture
\end{figure}

Now we look at the tree $\mathbb{T}_p$ as in Figure~2, but upside down,
so that $\varpi$ is at the bottom, and the tree branches upwards.
As in Remark \ref{rmk:metric graphs}, we consider it as a metric tree
where edges are intervals of length $1$, so that the Busemann function
of the tree becomes real\dash valued. Then we can consider the horocyclic
product with sliced hyperbolic plane. This is a situation where we
pair points $z \in \mathbb{H}_q$ and $w \in\mathbb{T}_p$ when
$\mathfrak{h}_q(z) - \mathfrak{h}(w)=0$ with ``$-$'' instead of ``$+$''
because of the opposite behaviour of the two functions mentioned above.

\begin{dfn} \label{def:treebolic} For integer $p \ge 2$ and real $q > 0$, 
\emph{treebolic space} is defined as
$$
\mathsf{HT}(p,q) = 
\{ \mathfrak{z}=(w,z) \in \mathbb{T}_p \times \mathbb{H}_q :
\mathfrak{h}(w) = \log_q(\Im z) \}.
$$
\end{dfn}

In these terms, treebolic space was introduced -- with notation 
$\mathsf{HT}(q,p)$ and elements $(z,w)$ in the place of $(w,z)$ --
by {\sc Bendikov, Saloff\dash Coste, Salvatori and Woess}~\cite{BSSW1} 
and studied in detail in \cite{BSSW2}. Previously, $\mathsf{HT}(p,p)$ (with 
integer $p \ge 2$) appeared in the work of {\sc Farb and Mosher}~\cite{FM1}, 
\cite{FM2}. 

To visualise $\mathsf{HT}(p,q)$, Figure~7 shows a compact portion
of that space in the case where $p=2$. To construct our space,
we need countably many copies of each of the lines 
$\mathsf{L}_k = \{ z \in \mathbb{H} :  \Im z = q^k \}$ and strips 
$\mathsf{S}_k = \{ z \in \mathbb{H} :  q^{k-1} \le \Im z \le q^k \}$,
where $k \in \mathbb{Z}$.  These copies are pasted together in a tree\dash like
fashion. To each vertex $v$ of $\mathbb{T}$, in treebolic
space there corresponds the \emph{bifurcation line} 
$\mathsf{L}_v = \{ v \} \times  \mathsf{L}_k\,$, where $k = \mathfrak{h}(v)$.
Attached below to the line $\mathsf{L}_v\,$, there is the copy
$$
\mathsf{S}_v =  \{ (w,z) : w \in [v^-,v]\,,\; z \in \mathsf{S}_k\,,\; 
\mathfrak{h}(w) = \log_q(\Im z)\}
$$
of $\mathsf{S}_k\,$. Attached above $\mathsf{L}_v\,$, there are the strips
$\mathsf{S}_u\,$, where $u$ ranges over the successor vertices of $v$ (i.e.,
$u^- = v$). 

\begin{figure}[h]
\hfill\beginpicture 

\setcoordinatesystem units <1.3mm,1.3mm>

\setplotarea x from -16 to 54, y from 0 to 48

\plot 0 0  40 20  48 36  8 16  0 0  -8 16  -2 28  38 48  36.666 45.333 /

\plot 48 36  54 48  14 28  8 16  2  28  42 48   44.4 43.2 /

\plot 4.8 22.4  -8 16  -14 28  26 48  28.4  43.2 /

\put{$\scriptstyle \bullet$} at 8 16
\put{$v$} [r] at 7 16
\put{$\leftarrow$ $\mathsf{S}_{v}$} [l] at 46 28
\put{$\leftarrow$ $\mathsf{L}_v$} [l] at 50 36
\put{$\leftarrow$ $\mathsf{S}_{u}\,,\; u^-=v$} [l] at 53 42

\put{Figure~7} at 70 0
\endpicture
\end{figure}

Thus, the sliced hyperbolic plane of Figure~6 is the front view of 
$\mathsf{HT}(p,q)$, while the upside\dash down version of the tree of 
Figure~2 is the side view. In the latter picture, every bi\dash infinite
geodesic $\pi(\varpi,\xi)$, where $\xi \in \partial^* \mathbb{T}_p\,$,
is the side view of one copy of $\mathbb{H}_q\,$. On each of those copies,
we have the standard hyperbolic metric. It extends to $\mathsf{HT}(p,q)$
as follows. 

Let $(w_1\,,z_1), (w_2\,,z_2) \in \mathsf{HT}$, and let $v=w_1 \curlywedge w_2$
(confluent with respect to $\varpi$, see Figure~2).
Then 
\begin{equation}\label{eq:HTmetric}
d_{\mathsf{HT}}\bigl((w_1\,,z_1),(w_2\,,z_2)\bigr) = 
	\begin{cases} d_{\mathbb{H}}(z_1\,,z_2)\,, \;\ 
		\text{if there is}\ \xi \in \partial^*\mathbb{T}\\
		 \qquad\qquad\quad \text{with}\  w_1, w_2 \in \pi(\varpi,\xi),\\
  \min
	\{ d_{\mathbb{H}}(z_1\,,z)+d_{\mathbb{H}}(z,z_2)
	: z \in L_{\mathfrak{h}(v)} \}\,,\\ 
                \qquad\qquad\quad \text{otherwise.}
\end{cases}
\end{equation}
Indeed, in the first case, $(w_1\,,z_1)$ and $(w_2\,,z_2)$ belong to the 
common copy of $\mathbb{H}_q\,$ whose side view is $\pi(\varpi,\xi)$. 
In the second case, $v$ is a vertex, and there are 
$\xi_1, \xi_2 \in \partial^*\mathbb{T}$  such that $\xi_1 \curlywedge \xi_2 =v$ and 
$w_i \in \pi(v,\xi_i)$, so that our points above the line $\mathsf{L}_v\,$ on two 
distinct hyperbolic planes that are glued together below $\mathsf{L}_v\,$: 
it is necessary to pass through some point $(v,z) \in \mathsf{L}_v$ 
on the way from $(w_1\,,z_1)$ to $(w_2\,,z_2)$. See Figure~8.

\begin{figure}[h]
\hfill\beginpicture 

\setcoordinatesystem units <1.3mm,1.3mm>

\setplotarea x from -16 to 54, y from -9 to 30

\plot 0 0  40 0  55 26  15 26  0 0  -18 23  13 23   /

\put{$\mathsf{L}_v$} [l] at 41.5 0
\put{\scriptsize $\bullet$} at 34 16
\put{$(w_2\,,z_2)$} [l] at 35 16
\put{\scriptsize $\bullet$} at -2 12
\put{$(w_1\,,z_1)$} [b] at -2 13
\put{\scriptsize $\bullet$} at 20.5 0
\put{$(v\,,z)$} [t] at 20.5 -1


\circulararc -10 degrees from -2 12 center at -10 -30

\circulararc 19 degrees from 34 16 center at 75 -33

\setdashes <2pt>
\circulararc -24 degrees from 5.4 9.9 center at -10 -30

\put{Figure~8} at 65 0
\endpicture
\end{figure}

Using this picture, one obtains an approximate analogue of
Lemma \ref{lem:daniela}.

\begin{lem}\label{lem:metric} {\rm \cite{BSSW2}.} For all 
$(z_1\,,w_1)$, $(z_2\,,w_2) \in \mathsf{HT}$,
with $\delta = \log(1+\sqrt 2)$,
$$
\begin{aligned}
d_{\mathsf{HT}}\bigl((w_1\,,z_1),(w_2\,,z_2)\bigr)
&\le d_{\mathbb{H}}(z_1\,,z_2) + (\log q)\, d_{\mathbb{T}}(w_1\,,w_2) - 
|\Im z_1 - \Im z_2|\\
&\le d_{\mathsf{HT}}\bigl((w_1\,,z_1),(w_2\,,z_2)\bigr) + 2\delta\,.
\end{aligned}
$$
\end{lem}
Let us now describe the isometry group. We already know the group
$\mathsf{Aff}(\mathbb{T}_p)$ of all automorphisms of the tree that
preserve the predecessor relation. (In terms of the action on the 
boundary, this is the group of all automorphisms of the tree which
fix $\varpi$.) On the other hand, consider the group of 
orientation\dash preserving isometries of $\mathbb{H}$ which send the
collection of all lines $\mathsf{L}_k$ to itself:
$$
\mathsf{Aff}(\mathbb{H}_q) 
= \left\{  g=\begin{pmatrix} q^n & b \\ 0 & 1 \end{pmatrix} :
n \in \mathbb{Z}\,,\;b\in \mathbb{R} \right\} \quad
\text{acting by} \quad gz = q^n z + b\,,\; z \in \mathbb{H}\,.
$$
Left Haar measure $dg$ and the modular function 
$\Delta_{\mathsf{Aff}(\mathbb{H}_q)}$ are given by
\begin{equation}\label{eq:modularAffRq}
dg = q^{-n}\,dn\,\,db \quad \text{and}\quad 
\Delta_{\mathsf{Aff}(\mathbb{H}_q)}(g) = q^{-n}\,,\quad\text{if}\quad
g={\scriptstyle \begin{pmatrix}q^n& b \\ 0 & 1 \end{pmatrix}}\,.
\end{equation}
Here, $dn$ is counting measure on $\mathbb{Z}$ and $db$ is Lebesgue 
measure on $\mathbb{R}$.
We can now consider the horocyclic product of $\mathsf{Aff}(\mathbb{T}_p)$
and $\mathsf{Aff}(\mathbb{H}_q)$. 
The following is not hard to prove; see \cite{BSSW2}, where the group is 
called $\mathcal{A}{(q,p)}$. 

\begin{thm}\label{thm:isogroup} The group
	\begin{align*}
	\mathcal{B} = \mathcal{B}(p,q) 
	& = \bigl\{ (g_1\,,g_2) \in \mathsf{Aff}(\mathbb{T}_p) \times \mathsf{Aff}(\mathbb{H}_q) : 
	\\&\qquad \log_{p} \Delta_{\mathsf{Aff}(\mathbb{T}_p)}(g_1) 
+ \log_{q} \Delta_{\mathsf{Aff}(\mathbb{H}_q)}(g_2) = 0 \bigr\} 
	\end{align*}
acts transitively on $\mathsf{HT}(p,q)$ by 
$$
(w,z) \mapsto (g_1w,g_2z). 
$$
It is the semi\dash direct product
$$
\mathbb{R} \rtimes \mathsf{Aff}(\mathbb{T}_p) $$ with respect to the action $$
 b \mapsto q^{\Phi(g_1)}\,b\,, \; g_1 \in \mathsf{Aff}(\mathbb{T}_p)\,,\; b \in \mathbb{R}\,,
$$
and it acts on $\mathsf{HT}(p,q)$ with compact quotient isometric with the circle of length
$\log q$. The full group of isometries of $\mathsf{HT}(p,q)$ is generated by 
$\mathcal{B}(p,q)$ and the reflection
$$
(w, x+ i\, y) \mapsto (w,-x+i\, y)\,.
$$
As a closed subgroup of $\mathsf{Aff}(\mathbb{T}_p) \times \mathsf{Aff}(\mathbb{H}_q)$, 
the group $\mathcal{B}(p,q)$ 
is locally compact, compactly generated and amenable, and its modular function is 
given by
$$
\Delta_{\mathcal{B}}(g_1,g_2)= (p/q)^{\Phi(g_1)}\,. 
$$
\end{thm}
Again, the full isometry group is non\dash unimodular and cannot have a discrete, co\dash compact
subgroup unless $q=p$.

Regarding the classification up to quasi\dash isometries, the following available result
is not as complete as Theorem \ref{thm:qi-DL} for DL-graphs. 

\begin{thm}\label{thm:qi-HT} {\rm\cite{FM1}}. 
Let $p, p' \ge 2$ be integers. Then $\mathsf{HT}(p,p)$ is quasi\dash isometric with
$\mathsf{HT}(p',p')$ if and only if $p$ and $p'$ are powers of a common integer.
\end{thm}

For the general case, there is the following working hypothesis, still
to be verified:\footnote{I thank David Fisher for an exchange on this issue.}
\begin{ques}\label{ques:qi-HT} Let $p, p' \ge 2$ be integers and $q, q' > 1$ real. 
 
Is it true that $\mathsf{HT}(p,q)$ is quasi\dash isometric with
$\mathsf{HT}(p',q')$ if and only if $p$ and $p'$ are powers of a common integer
and $\log p'/\log p = \log q'/\log q$?
\end{ques}

Again, there is a natural geometric compactification. Recall that the the 
boundary of $\mathbb{H}$ is $\mathbb{R} \cup \{{\boldinfty}\}$ in 
the upper half plane model. The compactification $\widehat{\mathbb{H}}$ 
of $\mathbb{H}$ is easier to visualise when one passes to the Poincar\'e 
disk model: it then is simply the closed unit disk. The boundary point 
${\boldinfty}$ then corresponds to the ``North pole'' $i$ (the 
imaginary unit), while $\mathbb{R}$ corresponds to the unit circle
without $i$. 

\begin{dfn}\label{def:HTbdry}
The geometric compactification $\widehat{\mathsf{HT}}(p,q)$ is the closure
of $\mathsf{HT}(p,q)$ in $\widehat{\mathbb{T}}_p\times \widehat{\mathbb{H}}_q\,$, 
and the boundary at infinity is
$$
\partial \mathsf{HT}(p,q) = \widehat{\mathsf{HT}}(p,q) \setminus \mathsf{HT}(p,q).
$$
\end{dfn}
Again, we can imagine the boundary as a ``filled 8'' as in Figure~5, but this time 
the second of the two disks making up the ``8'' is a true unit disk: the boundary is
$$
\partial \mathsf{HT}(p,q) = \Bigl(\widehat{\mathbb{T}}_p\times \{{\boldinfty}\}\Bigr)
\cup \Bigl(\{\varpi\}\times \widehat{\mathbb{H}}_q\Bigr)\,.
$$
The two pieces meet in the point $(\varpi,{\boldinfty})$.
Convergence of a sequence ${\boldz}_n =  (w_n\,,z_n) \in \mathsf{HT}(p,q)$ to the 
boundary is analogous to the case of $\mathsf{DL}$ (but recall that now the tree is a metric
graph, so that convergence of a sequence to a point in $\mathbb{T}$ does not require that
the sequence stabilises at that point):
When $w_n \to w \in \mathbb{T} \cup \partial^* \mathbb{T}$ then necessarily 
$z_n \to {\boldinfty}$, whence ${\boldz}_n \to (w,{\boldinfty})$.
In the same way, when $z_n \to z \in \mathbb{H} \cup \partial^* \mathbb{H}$, then 
${\boldz}_n \to (\varpi,z)$. Finally, it may also occur that $w_n \to \varpi$ and 
$z_n \to {\boldinfty}$, in which case
${\boldz}_n \to (\varpi,{\boldinfty})$. 

\bigskip
 
\noindent\textbf{C. Sol-groups, resp.\ manifolds}

\smallskip\noindent
We consider again the hyperbolic upper half plane 
$\mathbb{H} = \{ x + i\,w : x, w \in \mathbb{R}\,,\; w > 0 \}$,
but use a slightly different parametrisation and notation.
The standard lengh element in the $(x,w)$-coordinates is
$w^{-2}(dx^2 + dw^2)$.
We pass to the \emph{logarithmic model} by substituting
$z = \log w\,$, and in the coordinates $(x,z) \in \mathbb{R}^2$, the 
length element becomes $e^{-2z}dx^2 + dz^2$. Now we also change 
curvature to $-p^2$ by modifying the length element into
$$
ds^2 = d_{p}s^2 = e^{-2p z}\,dx^2+ dz^2\,.
$$
We write $\mathbb{H}(p)$ for the hyperbolic plane with this parametrization and
metric and ${\boldx}=(x,z)$ for elements of $\mathbb{H}(p)$, so that in
the upper half plane model, ${\boldx}$ corresponds to $x + i\,e^{pz}$. 

The function $\mathfrak{h}({\bf x}) = z$ 
is then (up to the scaling factor $\log p$) the Busemann function with respect 
to the boundary point ${\boldinfty}$. 
Thus, $\bigl(\mathbb{H}(p), \mathfrak{h}\bigr)$ is  a Busemann pair.
 
The affine group $\mathsf{Aff}\bigl(\mathbb{H}(p)\bigr) = 
\bigl\{ g = \bigl(\begin{smallmatrix} e^{p c} & a \\[1pt] 0 & 1\end{smallmatrix}\bigr) 
: a, c \in \mathbb{R} \bigr\}$ acts on $\mathbb{H}(p)$  by $g(x,z) = (e^{p c}x+b, a+z)$
as an isometry group. It modular function is $\Delta_p(g) = e^{-p a}$.

\begin{dfn} \label{def:sol} For $p, q > 0$, the horocyclic product of
$\mathbb{H}(p)$ and $\mathbb{H}(q)$ is the manifold
$$
\mathsf{Sol}(p,q) = \mathbb{H}(p) \times_{\mathfrak{h}} \mathbb{H}(q). 
$$
Topologically, it is $\mathbb{R}^3$, but the length element in 
the $3$\dash dimensional coordinates $(x,y,z)$ is
$$
ds^2 = d_{p,q}s^2 = e^{-2p z}\,dx^2 + e^{2q z}\,dy^2 + dz^2\,,
$$
with the procjections $(x,y,z) \mapsto (x,z) \in \mathbb{H}(p)$ and
$(x,y,z) \mapsto (y,-z) \in \mathbb{H}(q)$.
\end{dfn} 
It is harder to draw a reasonable picture than in the case of two trees. On should imagine
to replace the two trees in Figure~4 by two hyperbolic (upper half) planes, where the second
one is upside down. 

Regarding the analogues of lemmas \ref{lem:daniela} and \ref{lem:metric},
so far only the following inequality has been proved, see {\sc Brofferio, Salvatori and
Woess}~\cite{BSW}.

\begin{lem}\label{lem:metricSol} {\rm \cite[Proposition 2.8(iii)]{BSW}}.
 If $(x_1\,,y_1\,,z_1)$, $(x_2\,,y_2\,,z_2)$ $\in$ $\mathsf{Sol}(p,q)$, 
then
$$
	\begin{aligned}
	& d_{\mathsf{Sol}}\bigl((x_1\,,y_1\,,z_1),(x_2\,,y_2\,,z_2)\bigr) \\
	&\qquad \le d_{\mathbb{H}(p)}\bigl((x_1\,,z_1),(x_2,z_2)\bigr) + 
d_{\mathbb{H}(q)}\bigl((y_1\,,-z_1),(y_2,-z_2)\bigr) - |z_1 - z_2|\,.
\end{aligned}
$$
\end{lem}
It is an open (probably not too hard) exercise to derive a matching upper 
bound of the form 
$d_{\mathsf{Sol}}\bigl((x_1\,,y_1\,,z_1),(x_2\,,y_2\,,z_2)\bigr) + \text{\it const}\,$.

In the case of $\mathsf{Sol}$, the analogue of the isometry groups 
$\mathcal{A}(p,q)$ for $\mathsf{DL}$, resp.\ $\mathcal{B}(p,q)$ for 
$\mathsf{HT}$ is $\mathsf{Sol}$ itself. But in order to keep this 
analogy in mind, and also because we want to think of space and isometry 
group separately, we write $\mathcal{S}(p,q)$ for the corresponding
Lie group. 
\begin{fct}\label{fct:Solgroup} The Lie group
$$
\mathcal{S} = \mathcal{S}(p,q) = \Bigg\{
\mathfrak{g} = \begin{pmatrix} e^{p c} & a & 0         \\
	0         & 1 & 0         \\
	0         & b & e^{-q c}
\end{pmatrix}\,,\quad a, b, c \in \mathbb{R} \Bigg\}
$$
can be identified with $\mathsf{Sol}(p,q)$, such that $\mathfrak{g}$ 
as above corresponds to $(a,b,c)$. The (isometric, fixed\dash point\dash free) 
action on $\mathsf{Sol}(p,q)$ (or equivalently, the group product) is given by
$$
(a,b,c) \cdot  (x,y,z) = \bigl( e^{p c}x + a, e^{-q c}y + b, c+z\bigr)\,.
$$
The group is the horocyclic product of the two affine groups $\mathsf{Aff}\bigl(\mathbb{H}(p)\bigr)$
and $\mathsf{Aff}\bigl(\mathbb{H}(q)\bigr)$, consisting of all pairs $(g_1\,,g_2)$
in the product of those two groups which satisfy 
$$
\log_{p} \Delta_p(g_1) + \log_{q} \Delta_q(g_2) = 0\,.
$$
The modular function is 
$$
\Delta_{\mathsf{Sol}(p,q)}(\mathfrak{g}) = e^{(q-p)c}
\,,\quad \text{when}\quad \mathfrak{g} = \begin{pmatrix} 
e^{p c} & a & 0\\ 0 & 1 & 0 \\ 0 & b & e^{- q c} \end{pmatrix}.
$$
\end{fct}
Again, there is no co\dash compact lattice in $\mathcal{S}(p,q)$ unless $p=q$.
Regarding the quasi\dash isometry classification, we have the following analogue 
of Theorem \ref{thm:qi-DL}.
                                                                          
\begin{thm}\label{thm:qi-Sol}{\rm \cite{ESW1}+\cite{ESW2}}.
$\mathsf{Sol}(p,q)$ is quasi\dash isometric with
$\mathsf{Sol}(p',q')$ if and only if $p'/p = q'/q$.
\end{thm}

Once more, there is a natural definition of the boundary of $\mathsf{Sol}(p,q)$ when
we consider our manifold as a subspace of $\mathbb{H}(p) \times \mathbb{H}(q)$.
The boundary of $\mathbb{H}(p)$ is the upper boundary point ${\boldinfty}$ together 
with the real line at the bottom of the upper half plane, while in the logarithmic model,
the real line sits at $z = -\infty$. Anyway, it is better to think of the Poncar\'e disk
and its compactification $\widehat{\mathbb{H}}(p)$ as a closed disk (with the
proper scaling of the metric in view of the curvature parameters).  

\begin{dfn}\label{def:Solbdry}
The geometric compactification $\widehat{\mathsf{Sol}}(p,q)$ is the closure
of $\mathsf{Sol}(p,q)$ in $\widehat{\mathbb{H}}(p)\times \widehat{\mathbb{H}}(q)\,$, 
and the boundary at infinity is
$$
\partial \mathsf{Sol}(p,q) = \widehat{\mathsf{Sol}}(p,q) \setminus \mathsf{Sol}(p,q).
$$
\end{dfn}

The boundary looks once more like in Figure~5, but this time, both halves of the
``8'' are true full unit disks. This time, we omit the description of convergence
to the boundary, which is completely analogous to $\mathsf{DL}$ and $\mathsf{HT}$. 

At last, we mention the work of {\sc Troyanov}~\cite{Tro}, who has given a careful
description of various features of the geometry of $\mathsf{Sol}(1,1)$. This includes,
in particular, the \emph{visibility boundary.} Briefly spoken, it consists of
those boundary points which can be ``seen'' from the chosen origin (reference point)
in our space as the limit of a geodesic ray that starts at the origin and converges
to that boundary point. In case of $\mathsf{Sol}(p,q)$, as a subset of the geometric
boundary, the visibility boundary is the ``8'' without its interior points and 
without the point where the two circles meet.  Note that this is not the same
as in the visibility metric (where distance between geodesics is distance in unit sphere 
between their tangent vectors at 0) referred to in \cite{Tro}.\footnote{I thank Jeremie 
Brieussel for an exchange on this issue.}
The visibility boundary is
completely analogous for $\mathsf{DL}(p,q)$ and $\mathsf{HT}(p,q)$.

\bigskip

In this section, we have undertaken an effort to underline a variety of common
geometric (resp.\ group\dash theoretic) features of $\mathsf{DL}$, $\mathsf{HT}$ and 
$\mathsf{Sol}$ which become clear thanks to visualising these spaces as
horocyclic products.  

\section{Lamplighters and other discrete subgroups}

We now come to the second question of the title of this article.
So far, we have seen that when $p \ne q$, then none of the groups
$\mathcal{A}(p,q)$, $\mathcal{B}(p,q)$ and $\mathcal{S}(p,q)$
can contain co\dash compact lattices (discrete subgroups with compact quotient), 
and in particular, $\mathsf{DL}(p,q)$ is far from even resembling a 
Cayley graph. What happens when $p=q\,$?

During a  visit of R\"oggi M\"oller (Reykjavik) to Graz in
2000, we had discussed but not succeeded to prove that $\mathsf{DL}(p,q)$
is not quasi\dash isometric with a Cayley graph. Shortly later, he sent me
a letter (at that time, still on paper \& by classical mail!) telling that
in discussions with Peter Neumann they had realised that $\mathsf{DL}(p,p)$
\emph{is} a Cayley graph. Later we realised that this was the 
\emph{lamplighter group} over $\mathbb{Z}$. 

Let us start with an explanation in terms of graphs. Consider a finitely generated 
group $G$ (resp., for a picture, one of its Cayley graphs). Imagine that at each 
group element (vertex) there is a lamp. Each lamp can be in $p$ different states 
(off, or on in different colours or intensities) which are described by the 
set $\mathbb{Z}_p = \{0, \dots, p-1\}$ -- the cyclic group of order $p$. 
(We might take any other finite group.) We think of $G$, resp.\ its given
Cayley graph, as a street network, and imagine a lamplighter walking along. 
Initially, all lamps are off (state $0$),
and at each step the lamplighter can choose or combine the following actions:
walk from a crossroad (vertex) to a neighbouring one, and/or modify the state of
the lamp at the current position. After a finite number of steps, only finitely
many lamps will be on. To encode this process, we have to keep track of

\begin{itemize}\itemsep-\parsep

 \item the current position of the lamplighter -- an element $g \in G$ 
(graph vertex), and

 \item the current configuration of lamps -- a function $\eta: G \to \mathbb{Z}_p$
with finite support $\{ x : \eta(x) \ne 0 \}$.
\end{itemize}

Let $\mathfrak{C}$ be the collection of all finitely supported configurations.
It is a group with respect to elementwise addition mod $p$. We have to consider
all pairs $(\eta,g)$, where $g \in G$ and $\eta \in \mathfrak{C}$.
Now every $g \in G$ acts on $\mathfrak{C}$ by $L_g\eta(x) = \eta(g^{-1}x)$.
Thus, we have a semi\dash direct product, called the \emph{wreath product}
$$
\mathbb{Z}_p \wr G = \mathfrak{C} \rtimes G\,,\quad (\eta,g)(\eta',g') =
(\eta + L_g\eta', gg')\,.
$$
The same works when $\mathbb{Z}_p$ is replaced by any other group $H$, in which case
the above addition mod $p$ should be replaced with elementwise group operation in $H$.
Below, we shall always have $G = \mathbb{Z}$, in which case, for $k \in \mathbb{Z}$,
we have of course $L_k\eta(x) = \eta(x-k)$, and 
$(\eta,k)(\eta',k') = (\eta + L_k\eta', k+k')$. Wreath products are nowadays often called
\emph{lamplighter groups,} in particular when the base group is $G=\mathbb{Z}$.

The following figure illustrates an element of $\mathbb{Z}_2 \wr \mathbb{Z}$: 
the configuration $\eta$ is $=1$ at the~$\bullet$s, and the lamplighter stands at the $\circ$.

\begin{figure}[h]
\hfill\beginpicture 

\setcoordinatesystem units <5mm,5mm> 

\setplotarea x from -10.2 to 10.2, y from -2 to 1

\putrule from  -10.6   0 to 10.6  0
\putrule from  -10 -.1 to -10 .1
\putrule from  -9 -.1 to -9 .1
\putrule from  -8 -.1 to -8 .1
\putrule from  -7 -.1 to -7 .1
\putrule from  -6 -.1 to -6 .1
\putrule from  -5 -.1 to -5 .1
\putrule from  -4 -.1 to -4 .1
\putrule from  -3 -.1 to -3 .1
\putrule from  -2 -.1 to -2 .1
\putrule from  -1 -.1 to -1 .1
\putrule from  0 -.1 to 0 .1
\putrule from  10 -.1 to 10 .1
\putrule from  9 -.1 to 9 .1
\putrule from  8 -.1 to 8 .1
\putrule from  7 -.1 to 7 .1
\putrule from  6 -.1 to 6 .1
\putrule from  5 -.1 to 5 .1
\putrule from  4 -.1 to 4 .1
\putrule from  3 -.1 to 3 .1
\putrule from  2 -.1 to 2 .1
\putrule from  1 -.1 to 1 .1

\multiput {{$\bullet$}} at -10 0  -8 0  -7  0  -5 0
                                        -3 0  0 0  1 0  5 0  7 0  9 0 /

\put {{$\scriptstyle \bigcirc$}} at 1 0

\put{Figure~9} at 15 0
\endpicture
\end{figure}

We now explain the correspondence between the lamplighter group $\mathbb{Z}_p \wr \mathbb{Z}$
and the graph $\mathsf{DL}(p,p)$. For this purpose, let us again look at Figure~2. Given any
vertex of $\mathbb{T}_p\,$, we can label the edges to its successors from left to right with the digits
$0, \dots, p-1$ ($0, 1$ in Figure~2). We let $\Sigma_p$ be the collection of all 
sequences $\bigl(\sigma(n)\bigr)_{n \le 0}$ with finite support $\{ n : \sigma(n) \ne 0 \}$.
With a vertex $x$ we can then associate the sequence $\sigma_x \in \Sigma_p$ of the 
labels on the geodesic $\pi(\varpi,x)$ coming down from $\varpi$. 
Given $\sigma \in \Sigma_p$ and $k \in \mathbb{Z}$, there is precisely one vertex $x$ on the 
horocycle $H_k$ such that $\sigma_x = \sigma$. In other words, we have a bijection 
$$
T \leftrightarrow \Sigma_p \times \mathbb{Z}\,,\quad\text{where}\quad 
x \mapsto (\sigma_x,k)
$$
For example, the vertex $x$ in Figure~2 corresponds to $(\sigma,k)$, where $k = 0$ and
$\sigma = (\dots, 0,0,0,1,1)$.
In the above identification, the predecessor vertex of any $(\sigma,k)$ is $(\sigma',k-1)$, where 
$\sigma'(n) = \sigma(n-1)$ for all $n \le 0$.  

\smallskip

Now let $(\eta,k) \in \mathbb{Z}_p \wr \mathbb{Z}$. We split $\eta$ at $k$ by defining
$\eta_k^- = \eta|_{(-\infty\,,\,k]}$  and $\eta_k^+=\eta|_{[k+1\,,\,\infty)}$,
both written as sequences over the non\dash positive integers which belong to $\Sigma_p\,$:
$$
\eta_k^- = \bigl(\eta(k+n)\bigr)_{n \le 0} \quad\text{and}\quad 
\eta_k^+ = \bigl(\eta(k+1-n)\bigr)_{n \le 0}\,.
$$
Then $x_1 = (\eta_k^-,k)$ and $x_2 = (\eta_k^+,-k)$ are vertices, one in each of the two 
copies of $\mathbb{T}_p$ that make up $\mathsf{DL}(p,p)$. This yields the correspondence
between (the vertex set of) $\mathsf{DL}(p,p)$ and the lamplighter group 
$\mathbb{Z}_p \wr \mathbb{Z}$. It is a rather straightforward exercise to work out that
under this identification, our group acts transitively and without fixed points on
$\mathsf{DL}(p,p)$ and that the action preserves the neighbourhood relation of
the graph. 
%
See \cite{W-lamp}, where this is explained in more detail.

For $k \in \mathbb{Z}$ and $\ell \in \mathbb{Z}_p\,$,
let $\delta_k^{\ell} \in \mathfrak{C}$ be the configuration with value $\ell$ 
at $k$ and $0$ elsewhere. Then we can subsume the preceding
explanations as follows.

\begin{pro}\label{pro:ll} The lamplighter group 
$\mathbb{Z}_p \wr \mathbb{Z}$ embeds as a discrete, co\dash compact 
subgroup into the group $\mathcal{A}(p,p)$ of Proposition \ref{pro:af}. 
The Diestel\dash Leader graph $\mathsf{DL}(p,p)$
is the Cayley graph of the lamplighter group 
with respect to the symmetric set of generators
$$
\{ (\delta_1^{\ell},1)\,,\;(\delta_0^{\ell},-1) : \ell \in \mathbb{Z}_q \}\,.
$$
\end{pro}

This means that the actions of the lamplighter that correspond to crossing 
an edge in $\mathsf{DL}(p,p)$ are: ``either make first a step to the right and
then switch the lamp at the arrival point to any of the possible states,
or else first switch the lamp at the departure point to any of the possible
states and the make a step to the left.'' 

\smallskip

Regarding treebolic space and the group $\mathcal{B}(p,p)$ of 
Theorem \ref{thm:isogroup}, we have the following.

\begin{pro}\label{pro:BS} For integer $p \ge 2$, 
the \emph{Baumslag\dash Solitar group} 
$$
\textit{BS}(p) = \left\{ \begin{pmatrix} p^m & k/p^l \\ 0 & 1 \end{pmatrix}
: k, l, m \in \mathbb{Z} \right\} = \langle a, b \mid a\,b=b^{p}\,a \rangle
$$
embeds as a discrete, co\dash compact 
subgroup into the group $\mathcal{B}(p,p)$. 
\end{pro}

We omit the explanation; see \cite{FM1}, and in more detail \&
closer to the spirit of the present survey, \cite[\S 2]{BSSW2}.

\smallskip

Finally, we consider the $\mathsf{Sol}$ case and exhibit discrete, 
co\dash compact subgroups of $\mathcal{S}(p,p)$. We include an explanation
because this is so obvious to the specialists that it is not too easy
to find in the relevant literature.

Let $A = \bigl(\begin{smallmatrix} a & b \\ c & d \end{smallmatrix}\bigr) 
\in \textit{SL}_2(\mathbb{Z})$, an integer matrix with determinant $1$. We require 
that it has trace $a+d > 2$. Thus, it has  
eigenvalues $\lambda = \lambda(A) > 1$ and $1/\lambda$. Then $A$ induces an 
action of $\mathbb{Z}$ on $\mathbb{Z}^2$, such that $m \in \mathbb{Z}$
acts by
$$
\Bigl( {k \atop l} \Bigr) \mapsto A^m \Bigl( {k \atop l} \Bigr).
$$
This gives rise to the semi\dash direct product group
\begin{equation}\label{eq:sd}
\mathbb{Z}^2 \rtimes_A \mathbb{Z}  = \left\{
\begin{pmatrix} \;A^m & {\displaystyle {k \atop l}}\\[9pt] 0\quad0 & 1 \end{pmatrix} :
\; k,l,m \in \mathbb{Z} \right\}.
\end{equation}
We can find a matrix 
$\bigl(\begin{smallmatrix} \alpha & \beta \\ \gamma & \delta \end{smallmatrix}\bigr) 
\in \textit{SL}_2(\mathbb{R})$ that diagonalises $A$,
that is, 
$$
A  \begin{pmatrix} \alpha & \beta \\ 
                    \gamma & \delta \end{pmatrix}
= 
 \begin{pmatrix} \alpha & \beta \\ 
                    \gamma & \delta \end{pmatrix}
\begin{pmatrix} \lambda & 0 \\ 0 & \lambda^{-1} \end{pmatrix}.
$$
If we move the lattice $\mathbb{Z}^2$ by that matrix, then we end up in 
$\mathcal{S}(p,p)$, where $p=\log \lambda$. Indeed, conjugating with 
the  matrix
$$
B = \begin{pmatrix} \alpha & 0 & \beta & \\
       \gamma & 0 & \delta \\ 0 & 1 & 0 \end{pmatrix},
$$
we compute
$$
B^{-1} \begin{pmatrix} \;A^m & {\displaystyle {k \atop l}}\\[9pt] 0\quad0 & 1 
\end{pmatrix} B = 
\begin{pmatrix} e^{pm} & \delta k -\beta l  & 0 \\ 0 & 1 & 0 \\
       0 & - \gamma k + \alpha l & e^{-pm} \end{pmatrix}
$$
Note that $\bigl(\begin{smallmatrix} \alpha & \beta \\ \gamma & \delta \end{smallmatrix}\bigr)$ 
and $p$ cannot be chosen independently. Again, we subsume.

\begin{pro}\label{pro:SOL} For any matrix $A \in \textit{SL}_2(\mathbb{Z})$ 
with trace $>2$, the group $\mathbb{Z}^2 \rtimes_A \mathbb{Z}$ of \eqref{eq:sd}
embeds isomorphically into $\mathcal{S}(p,p)$ as a discrete, co\dash compact subgroup,
where $p = \log \lambda\,$, and $\lambda$ is the eigenvalue of $A$ with $\lambda > 1$.

Thus, the group acts on $\mathsf{Sol}(p,p)$ with compact quotient.
\end{pro}

\section{Further developments} \label{sec:else}

My own interest focusses on issues like random walks on graphs and groups, the associated
harmonic functions and the spectral theory of the corresponding transition operators,
resp.\ adjacency matrices. In case of non\dash discrete structures, it is natural to replace
random walks with variants of Brownian motion. What makes me most happy is when I
can use a good understanding of the geometry of the given structure to derive results
in this direction. 

Lamplighter groups have been of increasing interest
in the context of random walks since their first appearance in this field of research
in the seminal paper by {\sc Kaimanovich and Vershik}~\cite{KV}. Currently, insertion of
the word ``lamplighter'' in {\it MathSciNet} yields a response of 44 articles.

Realising the classical lamplighter groups $\mathbb{Z}_p \wr \mathbb{Z}$ in terms of
DL graphs enhanced the interest to study random walks on $\mathsf{DL}(p,q)$ for arbitrary
integers $p, q \ge 2$. The asymptotics in space and time of random walks on
$\mathsf{DL}(p,q)$ were first studied by {\sc Bertacchi}~\cite{Ber}. 

Regarding random walk on $\mathbb{Z}_p \wr \mathbb{Z}$ -- corresponding
to simple random walk on $\mathsf{DL}(p,p)$ -- without using the $\mathsf{DL}$ 
description, {\sc Grigorchuk and \.Zuk}~\cite{GZ}
were the first to show that the spectrum is pure point, when $p=2$, then generalised to
arbitrary $p$ by {\sc Dicks and Schick}~\cite{DS}. While pure point spectrum (that is,
the given self\dash adjoint operator admits a complete orthonormal system of -- typically
finitely supported -- eigenfunctions) is familar in the context of fractals,
this was the first example of this type regarding an infinite, finitely generated group.
Using the horocyclic product structure, {\sc Bartholdi and Woess}~\cite{BaWo} 
provide a direct, explicit construction of the spectrum of $\mathcal{A}$-invariant
nearest neighbour random walk on arbitrary $\mathsf{DL}(p,q)$. It is again pure point,
and it can be used to determine the exact asymptotics of return probabilities (see also
{\sc Revelle}~\cite{Rev} for these asmyptotics on the lamplighter group).

The other issue that could be treated in a rather complete way by using the horoycyclic
product geometry concerns positive harmonic functions and the Martin boundary,
see {\sc Brofferio and Woess}~\cite{W-lamp}, \cite{BW1}, \cite{BW2}. 

A very similar approach, though comprising several different technical details,
applies to Brownian motion on the two sister structures, $\mathsf{Sol}$ and
$\mathsf{HT}$ -- with increasing level of difficulty. For those two, spectrum as well
as Martin boundary are not yet determined rigorously, although the $\mathsf{DL}$
case of \cite{BW1} leads to very clear ideas how the Martin compactification should
look like: in the drift\dash free case it should be the respective geometric compactification,
as described in \S \ref{sec:horo}, while otherwise it should be its refinement in
terms of horo\dash levels; compare with \cite{BW1}. One should also mention here the
recent work on the harmonic measure of discrete time random walks on $\mathsf{Sol}(1,1)$
by {\sc Brieussel and Tanaka}~\cite{BrT}. 

While $\mathsf{Sol}$ has a smooth structure, treebolic space has singularities along all
the bifurcation lines. This makes the rigorous construction of a Laplace operator
(with vertical drift term) and the associated Laplace operator considerably harder,
see \cite{BSSW1}. Once this is achieved, still with some additional difficulties in view of
the spatial singularities, one can proceed in a similar spirit as for random walk on
$\mathsf{DL}$-graphs. The results on $\mathsf{HT}(p,q)$
concern once more rate of escape, central limit theorem,
convergence to the boundary and positive harmonic functions. See \cite{BSSW2}, 
\cite{BSSW3}.

One common feature in all three cases is that every positive harmonic function $f$ for
the respective transition, resp.\ Laplace operator decomposes as
$$
f(x_1\,,x_2) = f_1(x_1) + f_2(x_2)\,,
$$ 
where $x_1$ and $x_2$ are the ``coordinates'' 
(with $\mathfrak{h}(x_1) \pm \mathfrak{h}(x_2) = 0$)
in the two factors of the horocyclic product, and each $f_i$ is a non\dash negative 
harmonic function for the projection of the respective
operator on the respective factor in that product.

\smallskip

Of course, there are more general types of horocyclic, resp.\ horospherical 
products than the tree sister structures of \S \ref{sec:horo}. One is
the horocyclic product of more than $2$ trees, 
$$
\mathsf{DL}(p_1, \dots, p_d) 
= \{ (x_1, \dots, x_d) \in \mathbb{T}_{p_1}\times \dots \times \mathbb{T}_{p_d} :
\mathfrak{h}(x_1) + \dots +\mathfrak{h}(x_d) =0 \}\,,
$$
equipped with a suitable neighbourhood relation. Automorphism group,
spectrum, Poisson boundary and other issues have been studied by
{\sc Bartholdi, Neuhauser and Woess}~\cite{BNW}. Again, the spectrum is pure point, and 
again, $\mathsf{DL}(p_1, \dots, p_d)$
is not a Cayley graph when the $p_i$ do not coincide. For three trees,
$\mathsf{DL}(p,p,p)$ is a Cayley graph of a finitely presented lamplighter\dash like group 
that has also been
studied by {\sc Cleary and Riley}~\cite{CR}. (The dead\dash end property studied in that 
paper and its predecessor by {\sc Cleary and Taback}~\cite{CT} becomes immediately
clear when one realises these groups in terms of $\mathsf{DL}$ graphs.)
For $d \ge 4$ factors, in \cite{BNW} a large number of cases is determined where 
$\mathsf{DL}(p, \dots, p)$ \emph{is} a Cayley graph.
The smallest case when this is not known is $\mathsf{DL}(2,2,2,2)$, while
$\mathsf{DL}(p,p,p,p)$ is shown to be a Cayley graph for all odd $p \ge 3$.

Given a tree with degree $p+q$, where $p, q \ge 2$, one can draw it such that
every vertex has $p$ predecessors and $q$ successors. It also has a natural
level function $\mathfrak{h}$, which in reality is not the Busemann function with
respect to some boundary point. One can then consider the horocyclic product
with ``sliced'' hyperbolic plane $\mathbb{H}_r$ (where $1 < r \in \mathbb{R}$) to obtain
a version of treebolic space where the strips ramify in both vertical directions. 
When $p$ and $q$ are relatively prime and $r$ is chosen appropriately,
the \emph{non\dash amenable Baumslag\dash Solitar group} $\langle a, b \mid a\,b^{q}=b^{p}\,a \rangle$
acts on that horocyclic product as a discrete isometry group and with compact
quotient. This fact is used by {\sc Cuno and Sava}~\cite{CS} in order to determine
the Poisson boundary of random walks on that group.

\smallskip

Finally, {\sc Kaimanovich and Sobieczky}~\cite{KS1}, \cite{KS2} have constructed
horocyclic products of random trees and studied random walks in the resulting
random environment.

\smallskip

Many further interesting classes of horocyclic, resp.\ horospherical products
are at hand and waiting for future exploration. In conclusion, let me come back
to a new formulation of the question posed at the beginning, apparently still
open:

\begin{quote}
\emph{Is there a (connected, locally finite, infinite) vertex\dash transitive graph 
\emph{with unimodular automorphism group} that is 
not quasi\dash isometric with some Cayley graph?}
\end{quote}

\let\otb=\thebibliography
\def\thebibliography#1{\otb{#1}\small\itemsep-\parsep}
   
\bibliographystyle{amsplain}
\newpage
\bibliography{woess-imn-ref}

\vfill
{\it Author's address:
Institut f\"ur Mathematische Strukturtheorie (Math C),\\
Tech\-ni\-sche Universit\"at Graz,
Steyrergasse 30, A-8010 Graz.\\
{\email woess@tugraz.at}.}

\Artikelende

\ifalleinez \end{document}\fi